\documentclass[10pt, oneside]{article}
\usepackage{amsfonts, amsmath, amssymb, amscd, latexsym}
\usepackage{geometry}                
\usepackage{graphicx}
\usepackage{amssymb}
\usepackage{epstopdf}
\usepackage{hyperref}
\usepackage[utf8]{inputenc}
\usepackage{tikz}
\usepackage{mathtools}
\usepackage{cleveref}
\usepackage{calrsfs}
\usepackage[cal=boondoxo]{mathalfa}
\usepackage{amsthm}

\usepackage{comment}
\usepackage{breqn}
\usepackage{nth}

\usepackage{url}

\newtheorem{thm}{Theorem}
\newtheorem{lem}[thm]{Lemma}
\newtheorem{cor}[thm]{Corollary}
\newtheorem{prop}[thm]{Proposition}

\theoremstyle{definition}
\newtheorem{definition}{Definition}

\theoremstyle{definition}
\newtheorem*{obs}{Observation}
                                                                                  
\newcommand{\Z}{\mathbb{Z}}

\newcommand{\K}{\mathbb{K}}
\newcommand{\eme}{\mathcal{m}}
\newcommand{\F}{\mathbb{F}}

\author{Francisco Franco Munoz}
\date{}							
\title{\bf Subrings of Finite Commutative Rings}

\begin{document}

\newcommand{\Addresses}{{
  \bigskip
  \footnotesize

Francisco Franco Munoz, \textsc{Department of Mathematics, University of Washington,
    Seattle, WA 98195}\par\nopagebreak
  \textit{E-mail address}: \texttt{ffm1@uw.edu}
}}

\maketitle

\begin{abstract} We record some basic results on subrings of finite commutative rings. Among them we establish the existence of Teichmuller units and find the number of maximal subrings of a given finite local ring.
\end{abstract}

\section{Introduction}

 For a given ring $R$, a natural question is to determine \emph{all} its subrings. To expect a reasonable answer, we must restrict attention to a well behaved class of rings, for example finite rings. It's also possible to enlarge this class slightly, while still retaining many feature of finite rings. In some sense, local rings are such a class. After setting up some elementary results on finite rings, we focus on a class in between, namely, the class $\bigstar$ (see section \ref{section-4}) of those local rings of characteristic $p^N$ where the maximal ideal is nilpotent. This class retains most good features of finite local rings. We obtain by elementary methods the existence of ``coefficient rings", via the existence of ``Teichmuller units", and proceed to obtain our main result, Theorem \ref{main-thm} on the parametrization of maximal subrings with equal residue field. We conclude with a the calculation of the number of maximal subrings of finite local rings.

\subsection{Notation}

All the rings considered are commutative and unital, and all the homomorphisms are unital. 

By a local ring $(R, \eme_R)$ we mean a ring with a unique maximal ideal ($\eme_R)$, not necessarily Noetherian (sometimes called quasi-local rings). 

For a ring $R$ denote $R^{\times}$ the group of units of $R$.

\subsection{On previous literature}

We have aimed to provide self-contained proofs of most results. Some basic results of commutative algebra are assumed and we'll often refer to $\cite{SP}$ for comprehensive treatments of the matter at hand.


\section{Reduction to finite local}

\subsection{From finite index subrings to finite rings}

Here's a well known result linking finite index subrings with finite rings. For lack of reference we include the proof.

\begin{thm} Suppose $R$ is any ring. Then the proper finite index subrings of $R$ correspond with subrings of residue rings $R/I$ where $I$ runs over the collection of proper finite index ideals of $R$. Similarly for algebras over a field $\K$, replacing finite index with finite codimension. 
\end{thm}

\begin{proof} The nontrivial part consists of exhibiting a finite index ideal contained in a given proper finite index subring $S\subseteq R$. For that, consider the $S$-module map $R\to R/S$. This map gives a ring homomorphism $ev: R\to End_{S}(R/S)$ given by $ev(r)([s])=[rs]$. Here $End_{S}(R/S)$ is the ring of $S$-endomorphisms of the finite (set) $S$-module $R/S$, hence a finite abelian group. Its kernel $I$ is a finite index ideal in $R$ that clearly is contained $S$. In fact, $I$ is the ``conductor" of $S$ in $R$, which is the largest ideal of $R$ inside $S$. And $S/I\subseteq R/I$, the latter is a finite ring.

The case of $\K$-subalgebras is similar.
\end{proof}

\subsection{Basic reduction}

From the above, to study finite index subrings a given ring we only need finite rings.\\

Here's a collection of basic facts.
\begin{thm} Suppose $A$ is a finite ring.
\begin{enumerate}
\item $A=\prod A[p]$ is the direct product of subrings $A[p]$ which are its Sylow $p$-subgroups (as abelian group). And for any subring $B=\prod B[p]$, $B[p]\subseteq A[p]$.
\item A finite $p$-ring (i.e. whose additive group is a $p$-group) is up to isomorphism the direct product of local subrings in a unique way. 
\item A finite local ring is a $p$-ring for some prime $p$ and it's characteristic is $p^N$ for some $N\geq 1$.
\item One can build up all the subrings from products of local subrings of a product ring in a prescribed manner.
\item Consider a product of $p$-local finite rings $A=A_1\times A_2\times .... \times A_n$. The collection of maximal local subrings of $A$ is in bijection with the collection of maximal local subrings of the product of fields $\F_{q_1}\times ... \times \F_{q_n}$, where $\F_{q_i}$ is the residue field of $A_i$ and $q_i=p^{e_i}$ is a power of $p$.
\item Consider all the finite fields $\F_{q_i}$ as contained in a fixed algebraic closure of $\F_p$. The maximal local subrings of the product $\prod \F_{q_i}$ are contained in the product $\F^n$ where $\F$ is the intersection of all the $\F_{q_i}$.
\end{enumerate}
\end{thm}

Note: This reduces the study to that of subrings of finite local rings (and these are local). Most results are valid in the context of general Artinian rings with suitable modifications.

\begin{proof}

\begin{enumerate}

\item Clear.
\item This holds for any Artinian ring (\cite[Chapter 8]{Atiyah}).
\item From 1. since a local ring is not a product of rings.
\item Here's the precise statement: Let $B\subseteq A$ be a subring, $B=B_1\times ...\times B_r,  A=A_1\times ... \times A_n$ and all the $B_j, A_i$ local. There exists a partition of $n$ into $r$ sets $C_1, ... , C_r$ such that $\displaystyle B_j\subseteq \prod_{i\in C_j} A_i$ holds. The proof follows by writing each idempotent $\epsilon_j$ corresponding to $B_j$ as a sum of the minimal idempotents $e_i$ of $A$, i.e. $\epsilon_j = \sum_{i\in C_j}e_i$. These exists and are unique since local rings have no nontrivial idempotents. This gives the required mapping $j\mapsto C_j$ which has the desired properties.
\item From above, we only need to consider local subrings of a product ring, say $B\subseteq A= A_1\times ... \times A_n$ with $B$ local. Notice that for a local ring $B$, the sum $B+ \prod \eme_i$ is also local since $\prod \eme_i$ is a nilpotent ideal. So maximal local subrings of $A$ contain the product ideal $\prod \eme_i$, and so they're in bijection with the maximal local subrings of $\prod \F_{q_i}$.
\item If $B\subseteq \prod \F_i$ is local, then $B/\eme_B \to \F_i$ is an injection, and so $B\subseteq \F^n$, where $\F = \bigcap \F_i$. 
\end{enumerate} 
\end{proof}

\section{Teichmuller units}\label{section-4}

We can and will consider a larger class of commutative rings than just local finite. We'll work with: $$\bigstar: (R, \eme) \text{ local of characteristic } p^N, N\geq 1 \text{ and residue field } \F_q  \text{ such that } \eme \text{ is a nilpotent ideal} $$

We don't assume $\eme$ finitely generated. Notice any such $(R,\eme_R)$ is complete (not necessarily Noetherian). We could more generally consider complete local rings, by using more structural results, like Cohen's structure theorem (see \cite[tag 0323]{SP}) but we'll avoid doing so to be more explicit and elementary.

\begin{prop} $\bigstar$ is satisfied by finite local rings.
\end{prop}

\begin{proof} In fact any Artinian local ring has a nilpotent maximal ideal (see \cite[Chapter 8]{Atiyah}).
\end{proof}

When $R$ is finite, the following is well-known. We include a proof in our case for completeness.
\begin{prop} Under $\bigstar$, $1+\eme$ is a (not necessarily finite) $p$-group, i.e, every element has finite order a power of $p$.
\end{prop}

\begin{proof} Let $I\subset R$ ideal. We have the short exact sequence $ 0\longrightarrow I \longrightarrow R \longrightarrow R/I \longrightarrow 0$ that induces an exact sequence of multiplicative groups $ 1\longrightarrow 1+I \longrightarrow 1+\eme_R \longrightarrow 1+\eme_{R/I} \longrightarrow 1$ (Recall that $1+\eme_R$ is a group for $(R, \eme_R)$ local). Assume for the moment that's true when $N=1$ and let's prove by induction on $N$. Suppose holds for $N$, and if characteristic of $R$ is $p^{N+1}$, using the exact sequence with $I=p^NR$, it's enough to show $1+p^NR$ is a $p$-group. But this is easy: $(1+p^Nx)^p=1+{p\choose 1}p^Nx+ ... + {p\choose p-1}(p^Nx)^{p-1}+(p^Nx)^p=1$, since the middle $p-1$ coefficients vanish by divisibility of the binomial coefficients and since $kN\geq N$ for $1\leq k\leq p-1$, and the last vanishes since $pN\geq 2N\geq N+1$. Now, when $N=1$, i.e. $R$ has characteristic $p$ then $(1+x)^{p^r}=1+x^{p^r}$ and since $x$ is nilpotent $x^{p^r}=0$ for large $r$. This completes the induction and the proof.
\end{proof}

\begin{prop}  The exact sequence $1\longrightarrow 1+\eme_R  \longrightarrow R^{\times} \longrightarrow {\F_q^\times} \longrightarrow 1$ splits. Moreover there's a unique subgroup of $R^\times$ isomorphic to $\F_q^{\times}$.
\end{prop}

\begin{proof} This is immediate if $1+\eme_R$ were finite, since $\F_q^\times$ has order $q-1$ which is prime to $p$. In the general case, one can argue using a Zorn's Lemma argument. The key is that a finitely generated (abelian) $p$-group is actually finite. Details omitted.
\end{proof}

We obtain $R^{\times} \cong {\F_q^\times}\times (1+\eme_R)$. Denote $R^\times=T(R)\cdot (1+\eme_R)$ as an internal direct product of two subgroups, where $T(R)$ is the (unique) subgroup of $R^\times$ isomorphic to $\F_q^{\times}$.

\begin{definition} $T(R)$ is called the group of Teichmuller units of $R$. 
\end{definition}

The rings considered here might not be finite, e.g. the ring $R=\Z/p^N\oplus M$ where $M$ is \emph{any} $\Z/p^N$-module with multiplication making $M$ have zero square, i.e. $(x, m)(y, n)=(xy, xn+ym)$; this satisfies $\bigstar$ since $\eme_R^i=(p\Z/p^N)^i\oplus p^{i-1}M$ and taking $i=N+1$, this is zero. 

But by the next propositions, the rings of class $\bigstar$ are not that far from being finite.

\begin{prop}\label{prop-6} The ring generated by $T(R)$ is finite. 
\end{prop}

\begin{proof} Since $T(R)$ is a finite group, the ring generated by $T(R)$ is the image under evaluation of the group ring $(\Z/p^N)[T(R)]$, which is a finite ring.
\end{proof}

\begin{prop} Let $S\subseteq R$ be generated (as ring) by a finite number of elements. Then $S$ is finite.
\end{prop}

\begin{proof} By possibly enlarging $S$, we adjoin the finite group $T(R)$ and we can assume $S$ is generated by $R_1$ (= subring generated by $T(R)$) and a finite number of nilpotent elements $x_1, ..., x_a$. We can enlarge this finite set by multiplying by the elements of $R_1$, obtaining a finite set $y_1, ... , y_b$ of nilpotent elements closed under multiplication by $R_1$. The set of finite products of the $y_i$ is finite, given that there's $m$ such that $y_i^m=0$ for all $i$, and so the subspace $V$ of $\eme_R$ spanned by them is finite (= finitely generated subgroup of the $p$-group $\eme_R$). Finally notice that $S$, the ring generated by $R_1$ and $y_i$, is the sum $S=R_1+V$ which is finite.
\end{proof}

\begin{obs} For $L\subseteq R$, one can consider the residue field of $L$ naturally as a subfield of the residue field of $R$. Indeed, both $L$ and $R$ are local rings whose maximal ideals are the set of nilpotent elements, and so $\eme_L=\eme_R\cap L$, and $L/\eme_L = L/(\eme_R\cap L) \subseteq R/\eme_R$ in a natural manner.
\end{obs}

For any subfield $\F\subseteq \F_q$, denote $\tau(\F^\times)$ the (unique) subgroup of $T(R)$ that maps to $\F^\times$. We have the following relations for subrings and subgroups: $L\subseteq R$ is a subring, $T(L)\subseteq T(R)$ is a subgroup and if $\F$ is the residue field of $L$, then $T(L)=\tau(\F^\times)$ under the identification above $\F= L/\eme_L\subseteq R/\eme_R=\F_q$.

\begin{lem} Two subrings $L_1$ and $L_2$ have same residue field iff $T(L_1)=T(L_2)$.
\end{lem}

\begin{proof} We have by above $T(L)=\tau((L/\eme_L)^\times)$. So the forward direction is clear and the backward direction follows from the fact that a field satisfies $\F = \F^\times \cup \{0\}$. 
\end{proof}

\begin{thm} Given a subfield $\F \subseteq \F_q$ there is a maximal subring $R^{[\F]}$ and a minimal subring $R_{[\F]} \subseteq R$ whose residue fields are $\F$. Moreover $R_{[\F]}$ is finite. 
\end{thm}

\begin{proof} Given a subfield $\F$, the maximal subring having that residue field is $R^{[\F]}$, the inverse image of $\F\subseteq \F_q$ under the map $R\to R/\eme_R$. Now, let $S$ be subring with residue field $\F$. Since $T(S)=\tau(\F^\times)$, and $T(S)\subseteq S$, $S$ contains the subring $R_{[\F]}$ generated as $\Z/p^N\Z$-module by $\tau(\F^\times)$ which only depends on $\F$. This is a ring since it's the image under evaluation of the group ring $(\Z/p^N)[\F^{\times}]$, and so it's finite. Clearly, $R_{[\F]}$ has residue field $\F$, and $R_{[\F]} \subseteq S$.
\end{proof}

 Hence the set of subrings of $R$ with fixed residue field $\F$ has a maximum element $R^{[\F]}$ and a minimum element $R_{[\F]}$.  

\begin{obs} 
\begin{itemize}
\item $R_{[\F_p]} = \Z/p^N\Z$, the base ring. 
\item $R^{[\F_p]}$ is the ring generated by $1$ and the maximal ideal $\eme_R$.
\item By definition $R^{[\F_q]}=R$. 
\item We'll determine the structure of $R_{[\F_q]}$ next.
\end{itemize}
\end{obs}


\subsection{The ring  $R_{[\F_q]}$}

 In this section we assume again that our rings are finite. We need part of the structure theory of Galois rings, which are the Galois extensions of $\Z/p^N$. Denote $R(N, n)$ the ``Galois" ring of characteristic $p^N$ and residue fields $\F_q$ with $q=p^n$. The following is a characterization of them (see \cite{Ganske}, \cite{MacDonald}):

\begin{thm} Let $R$ be a ring of characteristic $p^N$ whose maximal ideal is $pR$ ($R$ is an unramified $\Z/p^N$-algebra). Then $R$ is isomorphic to the Galois ring $R(N,n)$, where $q=p^n=\#(\F_q)$.
\end{thm}

We need the following known result, whose proof we include for completeness: 

\begin{prop}\label{prop-11} Let $(A,\eme)$ local of characteristic $p$, and $A/\eme = \F_q$ be the residue field. Then there's an embedding $\F_q\subseteq A$.
\end{prop}

\begin{proof} The Frobenius map is an $\F_p$-algebra endomorphism $\varphi: A\to A$. Since $\eme^{p^t}=0$ for large enough $t$, $Ker(\varphi^t)=\eme$, and consequently, $\F_q=A/\eme\subseteq A$.
\end{proof}

\begin{thm}\label{thm-12} Let $R$ be a ring that's generated as $\Z/p^N\Z$-module by $T(R)$. Then $R$ is isomorphic to the Galois ring $R(N,n)$, where $q=p^n=\#(\F_q)$.
\end{thm}

\begin{proof} By above we only need to show it's maximal ideal is $pR$, i.e., to show that $R/pR$ is a field. But $R/pR$ is also a ring that's generated by $T(R)$ as $\Z/p^N \Z$-module (in fact as vector space over $\Z/p\Z=\F_p$). So we only need to show such a ring is a field. So assume $R$ has characteristic $p$ generated as $\F_p$-vector space by $T(R)$. By proposition \ref{prop-11}, $R$ contains a copy of its residue field $\F_q$. But $T(R)\subseteq \F_q$ in this way and the $\F_p$-vector space generated is in $\F_q$, i.e. $R\subseteq \F_q$ so $\F_q=R$.
\end{proof}

\begin{thm}[Structure of the ring  $R_{[\F_q]}$] $R_{[\F_q]}\cong R(N,n)$.
\end{thm}

\begin{proof} $R_{[\F_q]}$ is generated as $\Z/p^N\Z$-module by $T(R)$.
\end{proof}


\subsection{Remark on coefficient rings}

The results of section 3.1 imply that $R_{[\F_q]}$ is a coefficient ring for $R$, whose existence is guaranteed for any complete local ring by the Cohen structure theorem (see \cite[tag 0323]{SP}). Here's the statement of what being a coefficient ring means:

\begin{cor} The ring $S=R_{[\F_q]}$ satisfies 
\begin{enumerate}
\item $\eme_S=S\cap \eme_R$
\item $S$ has the same residue field as $R$.
\item $\eme_S = pS$.
\end{enumerate}
\end{cor}

\begin{proof} The first holds for any subring, the second by definition and the third property $\eme_S=pS$ is part of the characterization of $S$ being a ``Galois ring" \cite{MacDonald}.
\end{proof}


\section{General existence of subrings}

Let $(R,\eme)$ be a local ring satisfying $\bigstar$ from here on.

\begin{definition} The Teichmuller set of $R$ defined by $\bar{T}(R)=T(R)\cup \{0\}$. 
\end{definition}

\begin{lem} Every element of $R$ can be written uniquely as $t+m$, with $t\in \bar{T}(R)$ and $m\in \eme$. In other words, the addition map $\bar{T}(R)\times \eme \to R $ is a bijection of sets. Furthermore, an element is invertible iff the $t\in \bar{T}(R)$ component is non-zero.
\end{lem}

\begin{proof} Existence: By construction of the Teichmuller units for $R$. Uniqueness: By construction as well (the map $T(R)\to (\F_q)^{\times}$ is an isomorphism). The rest is clear. 
\end{proof}

\begin{lem}\label{lem-16} Let $S$ be a subring with the same residue field as $R$. If $\eme_S= \eme$ then $S=R$.
\end{lem}

\begin{proof} Follows since $R=\bar{T}(R)+\eme$ and $\bar{T}(S)=\bar{T}(R)$.
\end{proof}

\begin{thm} Let $S\subseteq R$ a subring. If $S\neq R$ then $S+\eme^2\neq R$. In particular, if $S$ is a maximal subring with the same residue field as $R$, then $\eme^2\subseteq S$. 
\end{thm}
 
\begin{proof} Notice that the second assertion follows from the first. Indeed, if $S$ is a maximal subring with the same residue field as $R$, then $S+\eme^2$ also has the same residue field as $R$ and so $S\subseteq S+\eme^2\subseteq R$ and $S+\eme^2\neq R$ implies that $S=S+\eme^2$, so $\eme^2\subseteq S$.
\vspace{1 mm}

To prove the first assertion, by contradiction, suppose that $S\neq R$ is such that $S+\eme^2 = R$. We'll derive a contradiction from this fact. 
\vspace{1 mm}

Using the Teichmuller set for $S$, we have the equality of sets $\bar{T}(S)+\eme_{S}+\eme^2 = R$. And hence the set of non-invertible elements of $R$ is $\eme_S+\eme^2=\eme$. From Lemma \ref{lem-18} below, we have that $\eme_S=\eme$. Hence $\bar{T}(S)+\eme=R$, hence $\bar{T}(S)=\bar{T}(R)$ so $R$ and $S$ have the same residue field and from Lemma \ref{lem-16} above, we obtain $S=R$, contradiction.
\end{proof}

\begin{lem}\label{lem-18} Let $S$ be a subring of $R$ such that $\eme_S+\eme^2=\eme$. Then $\eme=\eme_S$.
\end{lem}

\begin{proof} By induction we'll show that $\eme^{n}=\eme_{S}^n+\eme^{n+1} $ for all $n\geq 1$. The case $n=1$ is given. Assume the case $n$, and so $\eme^{n+1}=\eme^n\eme=(\eme_S^n+\eme^{n+1})(\eme_S+\eme^2)=\eme_S^{n+1}+\eme_{S}^n\eme^2+\eme^{n+1}\eme_S+\eme^{n+3}\subseteq \eme_S^{n+1}+\eme^{n+2}\subseteq \eme^{n+1}$ so equality holds throughout. 
\vspace{1 mm}

Now, since $\eme$ is a nilpotent ideal, there's minimal $n$ such that $\eme^{n+1}=0$, from here working backwards one gets that $\eme^l=\eme_S^l$ for all $l\geq 1$.
\end{proof} 


\begin{lem} \label{lem-19} Let $I$ be a ideal contained in $\eme$. Then $\eme=I+\eme^2$ implies $I=\eme$.
\end{lem}

\begin{proof} Immediate from Nakayama lemma, or the proof method above. 
\end{proof}

\begin{definition} The characteristic module for $R$ is the module $V :=V_R=\eme/(\eme^2+pR)$.
\end{definition}

\begin{lem} There exists a maximal subring $S$ with the same residue field as $R$ iff $R_{[\F_q]}\neq R$.
\end{lem}

\begin{proof} Any such $S$ satisfies $R_{[\F_q]}\subseteq S\subseteq R$.
\end{proof}

\begin{lem}\label{lem-21} Let $S$ be a maximal subring with the same residue field as $R$. Then $S$ contains the ideal $\eme^2+pR$.
\end{lem}

\begin{proof} We know that $\eme^2\subseteq S$. To show that $pR\subseteq S$. But $R=\bar{T}(R)+\eme=\bar{T}(S)+\eme$ and $p\bar{T}(S)\subseteq S$ and $p\eme\subseteq \eme^2\subseteq S$.
\end{proof}

\begin{lem} Let $S$ be a maximal subring with the same residue field as $R$. Then $\eme_S$ is an ideal of $R$. Moreover $\eme^2+pR\subseteq \eme_S$. 
\end{lem}

\begin{proof}  $R=\bar{T}(R)+\eme=\bar{T}(S)+\eme$ and $\bar{T}(S)\eme_S\subseteq \eme_S$ and $\eme\eme_S\subseteq \eme^2\subseteq \eme_S$. The last inclusion follows from Lemma \ref{lem-21} above.
\end{proof}

\begin{prop} The characteristic module is a vector space over $\F_q$. Moreover, it's vector space structure is compatible with $T(R)$. Specifically, Let $t\in T(R), x\in \eme$, and the images $[t]\in R/\eme$,  $[x]\in V$, then the actions are related as follows: $[t].[x]=[tx]$. 
\end{prop}

\begin{proof} $V=\eme/(\eme^2+pR)$ is an $R$-module that's annihilated by $\eme$, hence a $R/\eme=\F_q$-vector space. For compatibility, notice that the map $\phi: \eme \to \eme/(\eme^2+pR)$ is an $R$-module map so $[tx]=\phi(tx)=t.\phi(x)=t.[x]$ but the action of $R$ factors through $R/\eme$ so $t.[x]=[t].[x]$ hence the result. 
\end{proof}

\begin{prop}\label{prop-24} The ideals of $R$ containing $\eme^2+pR$ correspond to in one-to-one manner with $\F_q$-subspaces of $V$, in a inclusion preserving way.
\end{prop}

\begin{proof} The only thing to show that is given a subspace $W$ insider $V$, the corresponding $I\subseteq m$ is an ideal, where $W=I/(\eme^2+pR)$. But $R=\bar{T}(R)+\eme$, and since $\eme^2\subseteq I$, $I\eme\subseteq \eme^2\subset I$. To show that $T(R)I\subseteq I$, notice that since $W$ is $\F_q$-stable, $\F_qW\subseteq W$ implies $T(R)I\subseteq I+\eme^2+pR\subseteq I$.
\end{proof}

\begin{prop}\label{prop-25} The set $\bar{T}(R)+pR$ is a subring with the same residue field as $R$.
\end{prop}

\begin{proof} It's enough to show that $\bar{T}(R)+pR_{[\F_q]}$ is a subring for then $\bar{T}(R)+pR=\bar{T}(R)+pR_{[\F_q]}+pR$ is the sum of a subring and an ideal, hence a subring. It's also clear that these rings have the same residue field as $R$. Now, $R_{[\F_q]}$ is the $\Z/p^N$-span of $T(R)$ and from subsection 3.1 above on the ring $R_{[\F_q]}$, we know its maximal ideal is generated by $p$, so $R_{[\F_q]} = \bar{T}(R)+pR_{[\F_q]}$ is indeed a ring.

\end{proof}

\begin{cor}\label{cor-26} For any ideal $I$ containing $\eme^2+pR$, the set $\bar{T}(R)+I$ is a subring with the same residue field as $R$.
\end{cor}

\begin{proof} By Proposition \ref{prop-25}, the set $\bar{T}(R)+pR$ is a subring and so $\bar{T}(R)+I=\bar{T}(R)+pR+I$ is the sum of subring and an ideal hence a subring.
\end{proof}

\begin{prop} There exists a maximal subring $S$ with the same residue field as $R$ iff $V\neq 0$.
\end{prop}

\begin{proof} If there's one, by Lemma \ref{lem-21}, we have $\eme^2+pR\subseteq \eme_S\neq \eme$, so $V\neq 0$. Conversely, suppose that $V\neq 0$. Then $V$ has (at least one) codimension one $\F_q$-subspace $W$ which corresponds to an ideal $I\subseteq m$ containing $\eme^2+pR$ by Proposition \ref{prop-24}. By Corollary \ref{cor-26}, $S:= \bar{T}(R)+I$ is a subring with the same residue field as $R$ and it's maximal in $R$ since $W$ is a maximal subspace of $V$.
\end{proof}

Here's the main result:

\begin{thm}\label{main-thm} Assume $V\neq 0$. There's a $1-1$ correspondence between codimension one subspaces of $V$ and maximal subrings of $R$ with the same residue field as $R$.
\end{thm}

\begin{proof} Indeed, those subrings are entirely determined by their maximal ideals, and by proposition \ref{prop-24} they correspond to codimension one subspaces of $V$. Conversely, by corollary \ref{cor-26}, such an ideal gives a subring with the same residue field as $R$. This is clearly a bijective correspondence.
\end{proof}

The easier converse of Theorem \ref{thm-12} is true:

\begin{prop}\label{prop-29} If $\eme=pR$ then $R$ is generated by $T(R)$ as $\Z/p^N\Z$-module, and $R$ is finite.
\end{prop}

\begin{proof} By induction one can write any nonzero element $x$ as $p^iu$, for $0\leq i< N$, and $u$ invertible. The exponent $i$ is unique, in fact, $i=N-j$, where $j$ is the minimal such that $p^jx=0$.We also have that the group of invertible elements is the (internal) direct product of $T(R)$ and $1+\eme$. Suffices to show the result for $u\in 1+\eme$. But $\eme=pR$, so $p^iu=p^i(1+py)=p^i+p^{i+1}y$ for  some $y\in R$ and doing the same for $y$, this process eventually stops since $p^N=0$. In particular, $R$ is finite.
\end{proof}

\begin{prop}\label{prop-30} Assume $V=0$. Then the only subring of $R$ with its same residue field is $R$.
\end{prop}

\begin{proof} $V=0$ iff $\eme=\eme^2+pR$. By Lemma \ref{lem-19}, $\eme=pR$. Proposition \ref{prop-29} above says that $R$ is generated as a $\Z/p^N\Z$-module by $T(R)$. In other words, $R$ is its own ``lower ring", $R_{[\F_q]}=R$, and so it's the unique subring with $\F_q$ as residue field. 
\end{proof}

\begin{cor} The number of maximal subrings with the same residue field as $R$ is $\displaystyle \frac{q^{\rho}-1}{q-1}$, where $\rho = \dim_{\F_q}(\eme/(\eme^2+pR))$ (assumed finite).
\end{cor}

\begin{proof} By Theorem \ref{main-thm} and Proposition \ref{prop-30}, this follows for any $\rho \geq 0$.
\end{proof}

\begin{cor}\label{cor-32} Let $S$ be a maximal subring with the same residue field as $R$. Then $S$ is maximal in $R$ iff the index of $S$ in $R$ is $p^n=q=\#(\F_q)$.
\end{cor}

\begin{proof} By Theorem \ref{main-thm}, maximal subrings $S$ with the same residue field as $R$ correspond to codimension one $\F_{q}$-subspaces of $V$, hence their indices are $q$.
\end{proof}

Now, we consider chains of subrings $R_l\subseteq R_{l-1}\subseteq ... \subseteq R$. If we have that $S\subseteq R$ has residue field $\F$, then $R_{[\F]}\subseteq S\subseteq R^{[\F]}$. Let's consider those $S$ such that $S$ has the same residue field as $R$. 

\begin{prop} Consider a chain $R_l=R_{[\F_q]}\subseteq R_{l-1}\subseteq R_{l-2}\subseteq ...\subseteq R=R_0$. This is a maximal chain iff the indices of consecutive rings are $[R_{i-1}:R_i]=q$ for all $i$.
\end{prop}

\begin{proof} This is a maximal chain iff $R_{i}\subseteq R_{i-1}$ is a maximal subring. And by Corollary \ref{cor-32} that happens iff the index is $[R_{i-1}:R_{i}]=q$.
\end{proof}

Another consequence for maximal subrings is the following result, giving a sharper inclusion than the one provided by index considerations alone:

\begin{prop} In a maximal chain $R_l=R_{[\F_q]}\subseteq R_{l-1}\subseteq R_{l-2}\subseteq ...\subseteq R=R_{0}$, we have that $p^kR\subseteq R_k$.  
\end{prop}

\begin{proof} By induction, for $k=1$, $R_1\subseteq R$ is maximal and $pR\subseteq R_1$ by Lemma \ref{lem-21}. Assumed valid for $k$, then $p^{k+1}R=pp^kR\subseteq pR_k\subseteq R_{k+1}$ since $R_{k+1}\subseteq R_k$ is maximal.
\end{proof}

 
\subsection{Case of finite local rings with residue field $\F_p$}

The results in this sections are all corollaries from the results before.\\
Let $(R, \eme)$ be a nontrivial (i.e. not equal to $\Z/p^N\Z$) finite local ring of characteristic $p^N$ whose residue field is $\F_p$. The $\Z$-span is denoted $\langle - \rangle$. Consider the additive subgroup $\eme^2 + \langle p \rangle \subset \eme$. Then:

\begin{enumerate}
\item Any subring $R$ of index $p$ contains the subgroup $\eme^2 + \langle p \rangle$. Conversely, any subgroup $R$ containing $1$ and $\eme^2$ is a unital subring.
\item The additive group $\eme / (\eme^2 + \langle p \rangle)$ is a vector space over $\F_p$ which is not zero. Let $\rho$ be its dimension so that $\rho>0$.
\item The number of subrings of $L$ of index $p$ is $\displaystyle \frac{p^{\rho}-1}{p-1}$. In particular $L$ has a subring of index $p$.  
\item A subring is maximal iff its index is $p$.
\end{enumerate}

\subsection{Case of finite local rings with larger residue field}
 
Let $(R, \eme)$ be a finite local ring of characteristic $p^N$ with residue field $\F_q$ ($q=p^n, n>1$).

\begin{thm} There are two kinds of maximal subrings of $R$: those that have the same residue field as $R$ and those that have residue field a maximal subfield of $\F_q$. Rings of the first kind all have index $q=p^n$ in $R$ and there are $\displaystyle \frac{q^{\rho}-1}{q-1}$ of them ($0$ exactly when $R$ is the Galois ring $R(N,n)$). Here $\rho = dim_{\F_q}(V)$ is the dimension of the characteristic module. Of the second kind, there are as many as maximal subfields of $\F_q$ and there's exactly one per subfield, and their indices match in this correspondence. Their total number is the number of prime divisors of $n$.
\end{thm}

\begin{proof} Let $S$ be a maximal subring.
\begin{itemize}

\item $S$ has residue field $\F_q$.

Those are classified above and there are $\frac{q^{\rho}-1}{q-1}$ of them.

\item $S$ has residue field $\F$, a proper subfield of $\F_q$.

Then $S\subseteq R^{[\F]}\neq R$, and so $S=R^{[\F]}=\pi^{-1}(\F)$ by maximality (where $\pi$ is the projection $\pi:R\to R/\eme=\F_q$). We claim that those are maximal exactly when $\F$ is a maximal subfield. Indeed, the necessity is obvious. So let's assume $\F$ is a maximal subfield of $\F_q$. Let $L$ be a subring properly containing $S$. Then $L$ has residue field contains $\F$, hence equal to $\F$ or $\F_q$. But the former case can't happen since $S=\pi^{-1}(\F)$. So necessarily $L$ has residue field $\F_q$. But then $L$ contains $S$ which contains $\eme$ ($=\pi^{-1}(0)$) and so $\eme_L=\eme$ and since $L$ has residue field $\F$, by Lemma 10, $L=R$. So $S$ is maximal. Now, by general isomorphism theorems, their indices match: $[R:S]=[\pi^{-1}(\F_q):\pi^{-1}(\F)]=[\F_q:\F]$. From basic Galois theory of finite fields, we know that the number of maximal subfields of $\F_{q}$ is equal to the number of prime divisors of $n$.
\end{itemize}
\end{proof}

\Addresses

\end{document}